\documentclass[reqno]{amsart}

\usepackage{amsfonts}
\usepackage{amssymb}

\usepackage{amscd}
\usepackage{pictexwd,dcpic}
\usepackage{graphicx}
\usepackage{tikz}
\usepackage{fullpage}
\usepackage{hyperref}
\hypersetup{
    colorlinks=true,
    linkcolor=blue,
    filecolor=magenta,
    urlcolor=cyan,
    citecolor=cyan,}
\usepackage[nameinlink,capitalize]{cleveref}
\usepackage{enumitem}

\title{Translational surfaces and iterated resultants}

\author{Matthew Weaver}
\address{School of Mathematical and Statistical Sciences, Arizona State University, Wexler Hall, Tempe AZ 85281}
\email{matthew.j.weaver@asu.edu}

\date{}

\newtheorem{thm}{Theorem}[section]
\newtheorem*{thm-nonum}{Theorem}
\newtheorem{prop}[thm]{Proposition}
\newtheorem{lemma}[thm]{Lemma}
\newtheorem{cor}[thm]{Corollary}
\numberwithin{equation}{section}

\theoremstyle{definition}
\newtheorem{rem}[thm]{Remark}
\newtheorem{set}[thm]{Setting}
\newtheorem{notat}[thm]{Notation}
\newtheorem{defn}[thm]{Definition}
\newtheorem{quest}[thm]{Question}
\newtheorem{ex}[thm]{Example}
\newtheorem{obs}[thm]{Observation}

\Crefname{thm}{Theorem}{Theorems}
\Crefname{ex}{Example}{Examples}

\def\P{\mathbb{P}}

\def\a{{\bf a}}
\def\b{{\bf b}}
\def\c{{\bf c}}

\def\A{{\bf A}}
\def\B{{\bf B}}
\def\C{{\bf C}}

\def\f{{\bf f}}
\def\g{{\bf g}}
\def\h{{\bf h}}
\def\K{\mathbb{K}}

\def\x{{\bf x}}

\def\dim{\mathop{\rm dim}}

\def\quot{\mathop{\rm Quot}}
\def\rk{\mathop{\rm rank}}

\def\bideg{\mathop{\rm bideg}}
\def\syz{\mathop{\rm syz}}

\def\Res{\mathrm{Res}}
\def\Syl{\mathrm{Syl}}

\usepackage[all]{xy}

\usepackage{nicematrix}

\begin{document}

\begin{abstract}
A translational surface is a tensor product surface constructed from two space curves by translating one along the other. These surfaces are common within geometric modeling and, since their description is parametric, it is desirable to obtain the implicit equation of such a surface. These surfaces have been studied thoroughly by Goldman and Wang in \cite{GW18syz}, where a particular set of syzygies was identified and shown to yield the implicit equation through an inhomogeneous resultant. As the method of \cite{GW18syz} may fail in the presence of ill-behaved basepoints of the parameterization, we offer an alternative method in this article using iterated homogeneous resultants. The algorithm presented here involves smaller Sylvester matrices overall, potentially resulting in faster computation, and succeeds in many instances where the method of \cite{GW18syz} cannot be applied.
\end{abstract}

\maketitle


\section{Introduction}

A classical problem within algebraic geometry is to find the implicit equations of a variety defined as the image of a rational map. This so-called \textit{implicitization problem} has gained attention in recent years for its applications to computer-aided design (CAD) in low-dimensional settings, when such a map defines a curve or surface. In the present article, we consider the implicitization of \textit{translational surfaces} defined by translating a rational space curve along another such curve. As these are some of the basic modeling surfaces used within geometric modeling (see e.g. \cite{Liu99,LV17,PDS20}), one is particularly interested in efficient and computationally effective ways to implicitize these surfaces.

One of the notable properties of translational surfaces within geometric modeling is their ease of control. Since such a surface is constructed from rational space curves $\f^*(s)$ and $\g^*(t)$, they admit parameterizations in terms of the rational functions defining these curves. Whereas several variations of translational surfaces have been studied within the literature, a natural definition is given in \cite{LV17} as $\h^*(s,t) = \frac{\f^*(s) + \g^*(t)}{2}$. With this formulation, the surface parameterized by $\h^*(s,t)$ is translation invariant, i.e. translating the curves $\f^*(s)$ and $\g^*(t)$ by a vector translates the surface by this same vector. Similarly, the surface parameterized by $\h^*(s,t)$ is invariant under a variety of rigid motions, hence one is able to manipulate the surface $\h^*(s,t)$ simply by adjusting its generating curves.

While this setting is amenable for most applications, it is often more convenient to consider \textit{homogeneous} parameterizations, so as to apply tools from commutative algebra and algebraic geometry. Write
$$\f^* = \left( \frac{f_1}{f_0},\frac{f_2}{f_0},\frac{f_3}{f_0} \right)\quad\quad\text{and}\quad\quad \g^* = \left( \frac{g_1}{g_0},\frac{g_2}{g_0},\frac{g_3}{g_0} \right)$$
to denote rational space curves as before, for polynomials  $f_0,f_1,f_2,f_3  \in \K[s]$ and $g_0,g_1,g_2,g_3\in \K[t]$ where $\K$ is an algebraically closed field of characteristic zero. Homogenizing, we consider surfaces generated by the projective space curves
$\f\,: \P^1\dashrightarrow \P^3$ and $\g\,: \P^1\dashrightarrow \P^3$ defined by homogeneous polynomials $f_0,f_1,f_2,f_3 \in \K[s,u]$ and $g_0,g_1,g_2,g_3 \in \K[t,v]$. With these homogenizations and the definition of $\h^*(s,t)$  as above, we obtain the \textit{tensor product surface} parameterization $\h: \P^1\times \P^1 \dashrightarrow \P^3$ given by
\begin{equation}\label{Trans Surf Defn}
    \h = (h_0,h_1,h_2,h_3) = (2f_0g_0, f_0g_1+f_1g_0, f_0g_2+f_2g_0,f_0g_3+f_3g_0)
\end{equation}
defined by \textit{bihomogeneous} polynomials in $R=\K[s,u,t,v]$. Indeed, setting $\bideg s,u =(1,0)$ and $\bideg t,v = (0,1)$, if $\deg f_i = m$ and $\deg g_i =n$, then $\h$ is a tensor product surface parameterized in bidegree $(m,n)$.

As noted, the image $X_\h$ of the bihomogeneous parameterization $\h: \P^1\times \P^1 \dashrightarrow \P^3$ is called a tensor product surface in $\P^3$. These surfaces have been studied to great length, and a variety of techniques have been used to produce their implicit equations 
(see e.g. \cite{AHW05,Botbol11,DS16,SSV14,SG17,SG18,Weaver25}). Within recent years, the most effective implicitization tools have emerged from the study of \textit{Rees algebras} and \textit{syzygies}. In particular, relating these techniques to the notion of \textit{moving curves} and \textit{moving surfaces} following a parameterization has become a fruitful approach to the problem \cite{Cox08,CGZ00,SC95,SGD97,SSQK94}. The difficulty for such a parameterization $\h$ however, is that its syzygy module typically has many generators, increasing complexity. Unlike the case for projective curves, $\syz(\h)$ is generally far from free, hence one often aims to identify a subset of distinguished syzygies sufficient to carry out the necessary computations.

For a translational surface $X_\h$ as above, an extensive study on the syzygies of $\h$ was carried out by Goldman and Wang in \cite{GW18syz}. Their study reveals a close connection between the syzygies of $\h$ and the syzygies of its generating curves $\f$ and $\g$, which are readily available as $\syz(\f)$ and $\syz(\g)$ are free. In particular, a distinguished set of syzygies is identified and used to compute the implicit equation, using the resultant of three moving planes following $\h$. However, this method is shown to fail in the presence of a particularly ill-behaved \textit{basepoint}, i.e. a point in $\P^1\times \P^1$ where the parameterization $\h$ is not defined. Moreover, for a surface parameterized in bidegree $(m,n)$, the inhomogeneous resultant of \cite{GW18syz} is the determinant of a $2mn\times 2mn$ matrix, which may be large if $\f$ and $\g$ are parameterized in high degree.

The objective of this paper is to provide an alternative to the method developed in \cite{GW18syz}, in order to find the implicit equation of a translational surface $X_\h$. We provide a method using iterations of the classical Sylvester resultant, in the homogeneous case, and analyzing its factors at each step. We note that similar processes have been previously employed \cite{BPF26,BM09,LM09,McCallum99}, however this strategy is quite amenable to the bigraded setting here, and naturally eliminates a set of indeterminates at each step. Whereas this method has its own limitations, the matrices involved are much smaller compared to those in \cite{GW18syz}, allowing for faster computation in many cases. Moreover, we show that the algorithm presented here may be further simplified, depending on how either of the generating curves embeds into $\P^3$. 

We briefly describe how this paper is organized. In \Cref{Prelim Section}, we recall the necessary background material from \cite{GW18syz} on translational surfaces, moving surfaces following their parameterizations, and characterizations of their basepoints. In \Cref{Iterated Resultant Section}, we present a method to implicitize a translational surface $X_\h$ by iterating the classical Sylvester resultant. We provide a criterion under which this algorithm produces the implicit equation and discuss sufficient conditions under which this criterion holds. Lastly, in \Cref{Plane curves and ruled surfaces section}, we consider natural geometric instances in which this criterion is satisfied automatically, depending on how either of the generating curves lies within $\P^3$. In particular, we show that this process can be made simpler if either $\f$ or $\g$ is a planar curve in $\P^3$, and that our method coincides with the method developed in \cite{GW18ruled} when $X_\h$ is a \textit{ruled} translational surface.


\section{Properties of translational surfaces}\label{Prelim Section}

We briefly review certain aspects of translational surfaces from the perspective of algebraic geometry, adopting all conventions and notation from \cite{GW18syz}. We recall some results and observations on the syzygies of $\h$, and the associated moving planes following this parameterization. Moreover, we discuss and describe the basepoints of $\h$.

\subsection{Translational surfaces and syzygies}

Write $\K$ to denote an algebraically closed field of characteristic zero. As before, a translational surface is a rational tensor product surface $\h:\, \P^1\times \P^1 \dashrightarrow \P^3$ constructed from two space curves $$\f:\, \P^1 \dashrightarrow \P^3 \quad\quad\text{and} \quad\quad\g:\, \P^1 \dashrightarrow \P^3.$$ 
  Assume that these maps are defined by homogeneous polynomials $f_0,f_1,f_2,f_3$ of degree $m$ in $\K[s,u]$ and homogeneous $g_0,g_1,g_2,g_3$ in $\K[t,v]$ of degree $n$, and we may assume that $\gcd(\f)=1$ and $\gcd(\g)=1$. The translational surface $\h:\, \P^1\times \P^1 \dashrightarrow \P^3$ generated by $\f$ and $\g$ is given by the polynomials $h_0,h_1,h_2,h_3 \in R=\K[s,u,t,v]$ where
$$h_i = f_0 g_i + f_ig_0,$$
hence we often write $\h = f_0 \g + g_0\f$.  Here $h_0,h_1,h_2,h_3$ are \textit{bihomogeneous} polynomials, where the bigrading on $R$ is given by $\bideg s, u =(1,0)$ and $\bideg t,v = (0,1)$. As such, with $\deg f_i = m$ and $\deg g_i =n$, we say that the image of $\h$ is a translational surface $X_\h$ parameterized in bidegree $(m,n)$.

As mentioned, we aim to deduce homological information on $\h$ and its syzygies. Rather than analyze the full syzygy module $\syz(\h)$, we aim to describe a smaller subset of syzygies, sufficient for our purposes. As noted in \cite{GW18syz,GW18ruled}, we may deduce such information from the syzygies of the generating curves $\f$ and $\g$. Consider the matrices
\begin{equation}\label{Mg and Ng definition}
M_{\g} = \begin{bmatrix}
    2g_0&g_1&g_2&g_3\\
    0&g_0&0&0\\
    0&0&g_0&0\\
    0&0&0&g_0
\end{bmatrix}\quad\quad\text{and}\quad\quad
N_{\g} = \begin{bmatrix}
    \frac{g_0}{2}&-\frac{g_1}{2}&-\frac{g_2}{2}&-\frac{g_3}{2}\\
    0&g_0&0&0\\
    0&0&g_0&0\\
    0&0&0&g_0
\end{bmatrix}
\end{equation}
and note that $\h = \f M_{\g}$. Moreover, since $M_\g N_\g = N_\g M_\g = g_0^2{\bf I}$, where ${\bf I}$ is the $4\times 4$ identity matrix, we have the following relation between the syzygies of a generating curve $\f$ and the syzygies of $\h$.

\begin{lemma}[{\cite[pg. 79]{GW18syz}}]\label{syz lemma}
With $\f$ and $\h$ above, we have the following.
\begin{enumerate}
    \item[(a)] If $\a\in \syz(\f)$, then $N_\g\a\in \syz(\h)$.

    \item[(b)] If ${\bf p}\in \syz(\h)$, then $M_\g{\bf p} \in \syz(\f)$.
\end{enumerate}
\end{lemma}
Likewise, by symmetry, one may construct the matrices $M_\f$ and $N_\f$ to relate the syzygies of $\g$ to the syzygies of $\h$. The benefit of this approach is that the syzygies of $\f$ and $\g$ are much more accessible. Moreover, the polynomials of $\f$ and $\g$ may be described in terms of these syzygies.

\begin{rem}\label{res f and g remark}
We note that the syzygy modules $\syz(\f)$ and $\syz(\g)$ of the generating curves are free modules over $A=\K[s,u]$ and $B=\K[t,v]$ respectively, which follows from Hilbert's syzygy theorem \cite[19.7]{Eisenbud}. Hence the ideals of $\f$ and $\g$ have graded free resolutions
$$0\rightarrow \bigoplus_{i=1}^3 A(-m-\mu_i) \overset{\varphi_\f}{\longrightarrow}A(-m)^4 \rightarrow (\f)\rightarrow 0 \quad\quad\text{and}\quad\quad 0\rightarrow \bigoplus_{i=1}^3 B(-n-\nu_i) \overset{\varphi_\g}{\longrightarrow}B(-n)^4 \rightarrow (\g)\rightarrow 0,$$
where $\mu_1+\mu_2+\mu_3=m$ and $\nu_1+\nu_2+\nu_3=n$. Moreover, by the Hilbert-Burch theorem \cite[20.15]{Eisenbud} we have that $f_i =(-1)^i\det (\varphi_\f)_i$ and $g_i =(-1)^i\det (\varphi_\g)_i$ for $i=0,1,2,3$, where $(\varphi_\f)_i$ and $(\varphi_\g)_i$ denote the $3\times 3$ submatrices obtained by deleting row $i$ of $\varphi_\f$ and $\varphi_\g$, respectively.
\end{rem}

As the syzygy module of the curve $\f$ is free, one commonly says that the columns of $\varphi_\f$ form a \textit{$\mu$-basis} of $\f$, and similarly for $\g$ with the columns of $\varphi_\g$. Moreover, from the determinantal nature of these curves, one also commonly says that $\f$ and $\g$ are the \textit{outer product} of their syzygy matrices $\varphi_\f$ and $\varphi_\g$, respectively. Whereas $\h$ possesses neither of these properties, following \Cref{syz lemma} and \Cref{res f and g remark} we note that the columns of $N_\g\varphi_\f$ are linearly independent, and $\h$ is \textit{nearly} determinantal.


\begin{lemma}{\cite[2.5]{GW18syz}}\label{minors of Ngabc}
For $N_\g$ and $\varphi_\f$ as in (\ref{Mg and Ng definition}) and \Cref{res f and g remark}, write $(N_\g \varphi_\f)_i$ to denote the $3\times 3$ submatrix obtained by deleting row $i$ of $N_\g \varphi_\f$. For $i=0,1,2,3$, we have $(-1)^i \det (N_\g \varphi_\f)_i = \frac{1}{2}g_0^2 h_i$. 
\end{lemma}

Following \Cref{minors of Ngabc}, the syzygies from the columns of $N_\g \varphi_\f$ are special, in a sense. Although the equations of $\h$ do not coincide with the minors of this matrix identically, their agreement up to multiple will be sufficient for most purposes. In particular, we note that these syzygies will correlate to notable \textit{moving planes} following the parameterization $\h$.

\subsection{Moving planes and surfaces}

We briefly recall the notion of \textit{moving surfaces} following a parameterization, and their connection to syzygies. We note that this technique has been widely used within the study of implicitization of curves and surfaces \cite{Cox08,CGZ00,SC95,SGD97,SSQK94}.

\begin{defn}
Writing $S=\K[w,x,y,z]$ to denote the coordinate ring of $\P^3$, a \textit{moving surface} of degree $d$ is a polynomial in $R\otimes_\K S = R[w,x,y,z]$ of the form
$$L=\sum_{i+j+k+l=d}a_{ijkl}w^ix^jy^kz^l$$
with $a_{ijkl}\in R=\K[s,u,t,v]$. Moreover, the moving surface $L$ is said to \textit{follow} the parameterization $\h$ if 
$$\sum_{i+j+k+l=d}a_{ijkl}h_0^i h_1^j h_2^k h_3^l =0.$$
\end{defn}

Notice that as the parameters $s,u,t,v$ assume different values, the coefficients $a_{ijkl}\in R$ vary as well. Thus the surface $V(L) \subseteq \P^3$ ``moves" as the parameters change, hence the terminology. In particular, when $d=1$, one obtains a \textit{moving plane} $L= a_0 w+a_1x + a_2 y+a_3 z$ which follows the parameterization if 
$$a_0 h_0+a_1 h_1 + a_2 h_2+a_3 h_3=0.$$
In other words, $L$ follows $\h$ if and only if $(a_0,a_1,a_2,a_3)$ is a syzygy of the ideal of $\h$. Similarly, the moving surfaces of degree $d$ that follow $\h$ correspond precisely to the syzygies of the $d$th power of the ideal of $\h$. 

 As noted, the columns of $N_\g \varphi_\f$ are special syzygies of $\h$, following \Cref{minors of Ngabc}. Writing $\a$, $\b$, and $\c$ to denote the columns of $\varphi_\f$, we may form three moving planes from the syzygies $N_\g \a$, $N_\g \b$, and $N_\g \c$, following the discussion above. Indeed, writing $\x = [w,x,y,z]$ to denote the coordinates of $\P^3$, we have that $\x N_\g \a$, $\x N_\g \b$, and $\x N_\g \c$ are moving planes following $\h$. Notice that, since these polynomials vanish at $\h$, these moving surfaces intersect $X_\h$ as the parameters vary. With this, one hopes to completely trace out $X_\h$ with moving surfaces, however care must be taken to avoid the \textit{basepoints} of the parameterization.

\subsection{Basepoints}

The \textit{basepoints} of $\h$ are the points in $\P^1\times\P^1$ where the parameterization is undefined, i.e. the locus $V(\h)$ where $h_0,h_1,h_2,h_3$ simultaneously vanish. Since $\gcd \f =\gcd \g =1$, it follows that $\gcd \h =1$ \cite[pg. 76]{GW18syz}, hence the base locus of $\h:\, \P^1\times \P^1 \dashrightarrow \P^3$ is a finite set of points. For a translational surface, there is a particularly simple description of this set of basepoints.

\begin{lemma}[{\cite[2.1]{GW18syz}}]\label{basepoint lemma}
A point $(s_0,u_0,t_0,v_0)$ in $\P^1\times\P^1$ is a basepoint of $\h$ if and only if $f_0(s_0,u_0) =0$ and $g_0(t_0,v_0)=0$.
\end{lemma}

It can be seen that translational surfaces tend to have many basepoints. Indeed, dehomogenizing by setting $u=v=1$, the polynomials $f_0(s,1)$ and $g_0(t,1)$ have exactly $m$ and $n$ roots respectively, counting multiplicity. From \Cref{basepoint lemma}, it then follows that $\h$ has at least $mn$ basepoints, counting multiplicity. While the multitude of basepoints tends to add obscurity to the problem, a useful consequence is that translational surfaces tend to have low implicit degree, due to the relation between these quantities established in \cite{AHW05}.

We note that the basepoints of $\h$ tend to be ill-behaved, imposing further restrictions. Moreover, we identify a class of particularly problematic basepoints where the syzygies $N_\g \a$, $N_\g \b$, and $N_\g \c$ simultaneously vanish. As noted in \cite[3.2]{GW18syz}, the existence of such a basepoint causes the implicitization method of \cite{GW18syz} to fail.

\begin{lemma}[{\cite[3.1--3.2]{GW18syz}}]\label{bad basepoint}
If $p=(s_0,u_0,t_0,v_0)$ is a basepoint of $\h$, then $\rk \,(N_\g\varphi_\f)(p)\leq 1$. Moreover, $\rk (N_\g\varphi_\f)(p) =0$ if and only if $\f(s_0,u_0) = \lambda \g(t_0,v_0)$ for some nonzero $\lambda\in \K$.   
\end{lemma}

In other words, $N_\g\varphi_\f$ vanishes precisely if the generating curves $\f$ and $\g$ intersect at a basepoint of $\h$. In this case, such a basepoint blows up to a curve lying on the surface $X_\h$. As previously mentioned, the presence of such a basepoint leads to the failure of the implicitization method in \cite[3.4]{GW18syz}, hence we are careful to identify them.

With the basepoints of $\h$ understood, we note that, away from them, the moving curves $\x N_\g \a$, $\x N_\g \b$, and $\x N_\g \c$ trace out the surface $X_\h$. We end this section with a criterion for when these moving planes intersect uniquely along $X_\h$.

\begin{prop}\label{Ngabc Moving surfaces intersect at Xh}
For a point $p=(s_0,u_0,t_0,v_0) \in \P^1\times\P^1$, if $g_0(t_0,v_0)\neq 0$ then the planes $\x(N_\g \a)(p)$, $\x (N_\g \b)(p)$, $\x (N_\g \c)(p)$ intersect at a unique point on $X_\h$.
\end{prop}

\begin{proof}
Notice that since $g_0(t_0,v_0)\neq 0$, the point $p$ is not a basepoint of $\h$ by \Cref{basepoint lemma}. In particular, following \Cref{minors of Ngabc}, we have that $\rk (N_\g\varphi_\f) (p) =3$ as at least one of its maximal minors is non-vanishing at $p$. Thus $\x (N_\g\varphi_\f)(p) =0$ has a unique solution $(w_0,x_0,y_0,z_0)$ up to unit multiple, i.e. there is a unique point in $\P^3$ at which $\x(N_\g \a)(p)$, $\x (N_\g \b)(p)$, $\x (N_\g \c)(p)$ intersect.

To see that such a point of intersection lies on $X_\h$, note that $(N_\g\varphi_\f) (p)$ is an almost square matrix with full rank. Hence such a solution must be a unit multiple of its signed minors, by Cramer's rule. Thus by \Cref{minors of Ngabc}, and since $g_0(t_0,v_0)\neq 0$, up to unit multiple we have that $(w_0,x_0,y_0,z_0)$ agrees with $\h(p)$, hence this unique point of intersection lies on the surface $X_\h$ in $\P^3$.
\end{proof}


\section{Iterated resultants and translational surfaces}\label{Iterated Resultant Section}

We now introduce the main setting of the paper, and aim to provide an alternative to the algorithm given in \cite{GW18syz} for translational surfaces. We note that this previous method relies on a specialized resultant of the three moving planes $\x N_\g \a$, $\x N_\g \b$, and $\x N_\g \c$ after setting $u=v=1$, as the implicit equation appears as a factor. However, as previously mentioned, this method fails in the presence of any ill-behaved basepoints as in \Cref{bad basepoint}, as this resultant then vanishes identically. Additionally, for a translational surface of bidegree $(m,n)$, this resultant is computed as the determinant of a $2mn\times 2mn$ matrix, which may be large depending on the degrees of the generating curves.

Our objective is to compute the implicit equation using \textit{iterated resultants} that do not vanish identically. Moreover, these resultants will concern the determinant of much smaller matrices throughout. Additionally, whereas the method of \cite{GW18syz} requires one to pass to affine coordinates as one of the main elimination steps, we aim to remain projective.

\begin{set}\label{main setting}
Let $\f$ and $\g$ denote two space curves in $\P^3$, defined by homogeneous polynomials $f_0,f_1,f_2,f_3\in \K[s,u]$ of degree $m$, and homogeneous $g_0,g_1,g_2,g_3\in \K[t,v]$ of degree $n$, with $\gcd \f = \gcd \g =1$. Let $\h$ denote the corresponding translational surface of bidegree $(m,n)$ generated by $\f$ and $\g$. 
\end{set}

As noted in \Cref{res f and g remark}, the ideal of $\f$ has a graded free resolution of the form
\begin{equation}\label{res of f}
0\rightarrow\bigoplus_{i=1}^3 A(-m-\mu_i)\overset{\varphi_\f}{\longrightarrow}A(-m)^4 \rightarrow (\f)\rightarrow 0
\end{equation}
where $A=\K[s,u]$ and $\mu_1+\mu_2+\mu_3 =m$. We may assume that $\mu_1\leq \mu_2 \leq \mu_3$ throughout, and write $\a$, $\b$, and $\c$ to denote the syzygies of $\f$ as the ordered columns of $\varphi_\f$ in (\ref{res of f}). In particular, recall that $\syz(\f)$ is free, hence the set $\{\a,\b,\c\}$ is linearly independent.

As before, the syzygies of $\f$ may be translated to syzygies of $\h$ following \Cref{syz lemma}. However, we note that multiplying by $N_\g$ may introduce common factors among the entries of each column, hence we are careful to remove extraneous factors. As such, we introduce a set of \textit{reduced} syzygies of $\h$, following similar conventions used in \cite{GW18ruled}.

\begin{notat}\label{ABC notation}
With $\a$, $\b$, and $\c$ as above, write $\gcd(N_\g \a)$ to denote the greatest common factor among the entries of $N_\g \a$. With this, let $\A = \frac{N_\g \a}{\gcd(N_\g \a)}$ and note that this is still a syzygy on $\h$, since $R$ is a domain. Similarly, consider $\gcd(N_\g \b)$ and $\gcd(N_\g \c)$, and write $\B = \frac{N_\g \b}{\gcd(N_\g \b)}$ and $\C = \frac{N_\g \c}{\gcd(N_\g \c)}$.
\end{notat}

Although removing common factors may seem inconsequential, its purpose will become evident once we formulate resultants from the moving planes associated with these syzygies. First however, we make note of behavior similar to that observed in \Cref{Prelim Section}, now with the reduced syzygies $\A$, $\B$, and $\C$ in place of $N_\g\a$, $N_\g\b$, and $N_\g\c$.

\begin{prop}\label{gcds divide g0}
With the polynomials $\gcd(N_\g \a)$, $\gcd(N_\g \b)$, and $\gcd(N_\g \c)$ in \Cref{ABC notation}, we have the following.
\begin{itemize}
    \item[(a)] Each of $\gcd(N_\g \a)$, $\gcd(N_\g \b)$, and $\gcd(N_\g \c)$ divides $g_0$. In particular, these polynomials belong to $\K[t,v]$.\vspace{1mm}
    
    \item[(b)] The product $\gcd(N_\g\a)\gcd(N_\g\b)\gcd(N_\g\c)$ divides $g_0^2$.
\end{itemize}
\end{prop}

\begin{proof}
We prove the first part for $\gcd(N_\g \a)$ as the result follows similarly for the remaining polynomials. Recall from \Cref{minors of Ngabc} that the minors of $N_\g \varphi_\f$ agree with $\frac{1}{2}g_0^2 \h$. Hence it follows that $\gcd(N_\g \a)$ divides $\frac{1}{2}g_0^2 h_i$ for $i=0,1,2,3$. However, since $\gcd \f =\gcd \g =1$, we note that $\gcd \h =1$ as well \cite[pg. 76]{GW18syz}. Hence any irreducible factor of $\gcd(N_\g \a)$ must divide $g_0$, and so we have that $\gcd(N_\g \a) \in \K[t,v]$ by degree considerations. Whereas every irreducible factor of $\gcd(N_\g \a)$ divides $g_0$, writing $a_0,a_1,a_2,a_3$ to denote the entries of $\a$, we have
\begin{equation}\label{Nga description}
N_\g\a = \begin{bmatrix}
\frac{g_0a_0-g_1a_1-g_2a_2-g_3a_3}{2}\\
g_0a_1\\
g_0a_2\\
g_0a_3
\end{bmatrix}
\end{equation}
and so $\gcd(N_\g \a)$ divides $g_0a_1$, for instance. However, since $\gcd(N_\g \a) \in \K[t,v]$ and $a_1 \in\K[s,u]$, it follows that $\gcd(N_\g \a)$ divides $g_0$. Similarly it follows that $\gcd(N_\g \b)$ and $\gcd(N_\g \c)$ divide $g_0$ which proves (a).

For part (b), again recall that the signed minors of $N_\g\varphi_\f$ are precisely $\tfrac{1}{2}g_0^2\h$. With this, it follows that $\gcd(N_\g\a)\gcd(N_\g\b)\gcd(N_\g\c)$ divides $g_0^2h_i$ for $i=0,1,2,3$. From the previous part, it is clear that $\gcd(N_\g\a)\gcd(N_\g\b)\gcd(N_\g\c)$ divides $g_0^3$. As such, and since $h_i=f_0g_i + f_ig_0$, it follows that the product $\gcd(N_\g\a)\gcd(N_\g\b)\gcd(N_\g\c)$ divides $g_0^2f_0g_i$. Noting that $f_0 \in \K[s,u]$ and the remaining polynomials involved belong to $\K[t,v]$, it follows that $\gcd(N_\g\a)\gcd(N_\g\b)\gcd(N_\g\c)$ divides $g_0^2g_i$ for $i=0,1,2,3$. Now if $\gcd(N_\g\a)\gcd(N_\g\b)\gcd(N_\g\c)$ does not divide $g_0^2$, then $g_0,g_1,g_2,g_3$ must share a common factor. However, this is impossible as $\gcd (\g) =1$.
\end{proof}

\begin{rem}\label{degrees ABC remark}
 Following \Cref{gcds divide g0}, notice that, within the bigrading of $R=\K[s,u,t,v]$, the entries of the syzygies $\A$, $\B$, and $\C$ have degree $\mu_1$, $\mu_2$, $\mu_3$ respectively, with respect to $s,u$, and degree at most $n$, with respect to $t,v$. As before, recall that we assume that $\mu_1\leq \mu_2\leq \mu_3$.   
\end{rem}

From \Cref{gcds divide g0} and \Cref{minors of Ngabc} it follows that the signed minors of $[\A\,\B\,\C]$ are precisely $\tfrac{1}{2} G\h$ where 
\begin{equation}\label{G defn}
G=\frac{g_0^2}{\gcd(N_\g\a)\gcd(N_\g\b)\gcd(N_\g\c)}
\end{equation} 
and we note that this is a polynomial in $\K[t,v]$ by \Cref{gcds divide g0}. With this, we may provide an analog of \Cref{Ngabc Moving surfaces intersect at Xh} for the moving planes $\x\A$, $\x\B$, $\x\C$.

\begin{prop}\label{Intersection xA and xB and xC}
For a point $p=(s_0,u_0,t_0,v_0) \in \P^1\times\P^1$, if $g_0(t_0,v_0)\neq 0$ then the planes $\x\A(p)$, $\x \B(p)$, $\x \C(p)$ intersect at a unique point on the surface $X_\h \subseteq \P^3$.
\end{prop}

\begin{proof}
Notice that, since $g_0(t_0,v_0)\neq 0$, $p$ is not a basepoint of $\h$ by \Cref{basepoint lemma}. As noted, from \Cref{minors of Ngabc} and \Cref{gcds divide g0}, it follows that the signed minors of $ [\A\,\B\,\C]$ are $\tfrac{1}{2} G\h$, for $G$ in (\ref{G defn}). Since $G$ divides $g_0^2$, if $g_0(t_0,v_0)\neq 0$ then $G(t_0,v_0)\neq 0$ as well, and so it follows that $ [\A\,\B\,\C](p)$ has rank 3. The claim then follows in a manner similar to the proof of \Cref{Ngabc Moving surfaces intersect at Xh}. 
\end{proof}

As previously noted, the syzygy module $\syz(\h)$ is not free. However, the syzygies $\A$, $\B$, $\C$ do share similarities with their counterparts $\a$, $\b$, and $\c$, which form a basis for $\syz(\f)$.

\begin{prop}\label{A B C linearly independent}
The syzygies $\A$, $\B$, $\C$ are linearly independent over $R=\K[s,u,t,v]$.
\end{prop}

\begin{proof}
    Suppose that there exist polynomials $\alpha,\beta,\gamma \in R$ such that 
    $$\alpha \A + \beta \B + \gamma \C = 0.$$
For $M_\g$ as in (\ref{Mg and Ng definition}), multiplying by the matrix $\gcd(N_\g \a) \gcd(N_\g \b)\gcd(N_\g \c) \cdot M_\g$ yields
$$\alpha \gcd(N_\g \b)\gcd(N_\g \c) g_0^2 \cdot  \a + \beta \gcd(N_\g \a)\gcd(N_\g \c) g_0^2 \cdot \b +\gamma \gcd(N_\g \a)\gcd(N_\g \b) g_0^2 \cdot \c = 0,$$
noting that $M_\g N_\g = N_\g M_\g = g_0^2{\bf I}$. However, recall that $\{\a,\b,\c\}$ is linearly independent over $\K[s,u]$, and hence over $R$. Moreover, since $g_0\neq 0$ and the greatest common factors are non-vanishing as well, it follows that $\alpha=\beta=\gamma =0$.
\end{proof}

\subsection{Iterated resultants}

We now begin the process of finding the implicit equation of $X_\h$ in $\K[w,x,y,z]$, the coordinate ring of $\P^3$. Writing $\x$ to denote these coordinates, we may form the moving surfaces $\x\A$, $\x\B$, and $\x\C$ as in \Cref{Prelim Section}. Since $\A$, $\B$, and $\C$ are syzygies on $\h$, these moving surfaces follow the parameterization.

We employ the use of homogeneous resultants throughout, i.e. the determinant of the classical \textit{Sylvester matrix}. For $f$ and $g$ homogeneous polynomials  in $A[x,y]$ of degree $m$ and $n$ respectively, where $A$ is a unique factorization domain, write
\begin{align*}
  f&=a_mx^m + a_{m-1}x^{m-1}y+\cdots + a_0y^m\\
  g&=b_nx^n + b_{n-1}x^{n-1}y+\cdots + b_0y^n.
\end{align*}
Recall that the Sylvester matrix of $f$ and $g$ with respect to $x,y$ is the $(m+n)\times (m+n)$ matrix
\begin{equation}\label{Sylvester matrix definition}
\Syl_{xy}(f,g)=
\left[
\begin{array}{ccccccc}
a_m & a_{m-1} & \cdots & a_0 &        &        &        \\
    & a_m     & a_{m-1}& \cdots & a_0 &        &        \\
    &         & \ddots &        &     & \ddots &        \\
    &         &        & a_m    & a_{m-1} & \cdots & a_0 \\[1ex]
b_n & b_{n-1} & \cdots & b_0 &        &        &        \\
    & b_n     & b_{n-1}& \cdots & b_0 &        &        \\
    &         & \ddots &        &     & \ddots &        \\
    &         &        & b_n    & b_{n-1} & \cdots & b_0
\end{array}
\right]
\begin{array}{l}
\left.\rule{0pt}{1.1cm}\right\}\;\text{$n$}\\[6pt]
\left.\rule{0pt}{1.1cm}\right\}\;\text{$m$}
\end{array}
\end{equation}
and $\Res_{xy}(f,g) = \det \Syl_{xy}(f,g)$ is the Sylvester \textit{resultant} of $f$ and $g$, with respect to $x,y$. For an introduction to resultants, we refer the reader to \cite[ch. 3]{CLO} and \cite{GKZ} for the more general theory. For our purposes we note that $\Res_{xy}(f,g)=0$ precisely when $f,g$ share a common factor in $A[x,y]$, or namely if $f$ and $g$ have a common zero in $\P^1_L$, where $L$ denotes the algebraic closure of the quotient field $\quot(A)$.

Resultants are a classical technique within elimination theory as they are expressions in the coefficients of the polynomials involved, eliminating their indeterminates. Moreover, provided they do not vanish identically, the resultant of two moving surfaces following a parameterization will be seen to follow it as well. As such, we employ this technique with the moving surfaces $\x\A$, $\x\B$, and $\x\C$, which share no common factor, by construction.

\begin{prop}\label{xA xB xC irreducible and coprime}
The polynomials $\x\A$, $\x\B$, and $\x\C$ are irreducible and pairwise coprime. 
\end{prop}

\begin{proof}
We verify the irreducibility of $\x\A$ as the claim then follows similarly for the remaining polynomials. Notice that, as a polynomial in $R[w,x,y,z]$, $\x\A$ is linear. Hence by degree considerations, if $\x\A$ is reducible, then it must admit a nontrivial factor in $R=\K[s,u,t,v]$. However, as $\{w,x,y,z\}$ is an $R$-linearly independent set, such a factor must then divide each of the coefficients of $\x\A$. However, these coefficients are precisely the entries of $\A$ which share no common factor by construction, hence no such nontrivial factor can exist. The second claim now follows from \Cref{A B C linearly independent}, noting that $\x\A$, $\x\B$, and $\x\C$ are irreducible.
\end{proof}


Recall that we assume that $\mu_1\leq \mu_2\leq \mu_3$ in (\ref{res of f}), and these are the degrees of $\x\A$, $\x\B$, $\x\C$ with respect to $s,u$ following \Cref{degrees ABC remark}. We begin with their homogeneous resultants with respect to $s,u$, and recall that these are the determinants of the Sylvester matrices of these moving planes, taken as polynomials in $s,u$ with coefficients in $\K[t,v,w,x,y,z]$.

\begin{prop}\label{Res su nonzero}
Both $\Res_{su}(\x\A,\x\C)$ and $\Res_{su}(\x\B,\x\C)$ are nonzero polynomials in $\K[t,v,w,x,y,z]$. Moreover, as moving surfaces, $\Res_{su}(\x\A,\x\C)$ and $\Res_{su}(\x\B,\x\C)$ follow the parameterization $\h$.
\end{prop}

\begin{proof}
Notice that, since $\mu_1\leq \mu_2\leq \mu_3$, we have that $\mu_3\geq 1$. As this is the degree of $\x\C$, with respect to $s,u$, it follows that neither Sylvester matrix $\Syl_{su}(\x\A,\x\C)$ nor $\Syl_{su}(\x\B,\x\C)$ is empty. With this, $\Res_{su}(\x\A,\x\C) =0$ only if $\x\A$ and $\x\C$ share a common factor, which is impossible by \Cref{xA xB xC irreducible and coprime}. Similarly, we have that $\Res_{su}(\x\B,\x\C) \neq 0$.

Now that $\Res_{su}(\x\A,\x\C)$ is seen to be a nonzero polynomial, it is clear that it follows $\h$, as a moving surface. Indeed, since both $\h\A=0$ and $\h\C=0$, the polynomials $\x\A$ and $\x\C$ clearly have a common factor when evaluated at $\h$, hence $\Res_{su}(\x\A,\x\C) (\h) = \Res_{su}(\h\A,\h\C) =0$. Similarly, it follows that $\Res_{su}(\x\B,\x\C)$ follows $\h$ as well.
\end{proof}

\begin{cor}\label{Res su factors} 
The resultants in \Cref{Res su nonzero} factor as $\Res_{su}(\x\A,\x\C) = E_0 F_0$ and $\Res_{su}(\x\B,\x\C) = E_1 F_1$ for polynomials $E_0, E_1 ,F_0, F_1 \in \K[t,v,w,x,y,z]$ where $F_0, F_1$ are irreducible polynomials and follow the parameterization $\h$, as moving surfaces.   
\end{cor}

\begin{proof}
As noted in \Cref{Res su nonzero}, both $\Res_{su}(\x\A,\x\C)$ and $\Res_{su}(\x\B,\x\C)$ are nonzero polynomials in $\K[t,v,w,x,y,z]$ following $\h$. As a moving surface follows a parameterization if and only if one of its irreducible factors does, the claim follows.
\end{proof}

We note that the additional factors $E_0$ and $E_1$ may further factor. It is not uncommon for the resultants above to be irreducible, in which case $E_0 =E_1=1$, however it is also possible that $E_0$, $E_1$ contain extraneous factors corresponding to basepoints of $\h$. In either case, we are primarily interested in the irreducible factors $F_0$, $F_1$ of $\Res_{su}(\x\A,\x\C)$ and $\Res_{su}(\x\B,\x\C)$ following $\h$. 

In general, removing unnecessary factors will be a common theme overall. Not only does this reduce the degrees of the polynomials involved, leading to smaller Sylvester matrices, but it also prevents unwanted vanishing. Indeed, recall that 
\begin{equation}\label{resultant product property}
\Res_{su}(p_1p_2, q) =\Res_{su}(p_1, q) \Res_{su}(p_2, q)   
\end{equation}
for polynomials $p_1$, $p_2$, and $q$. Hence removing factors from the polynomials involved reduces the number of factors of their resultant, one of which may potentially vanish. In particular, this is precisely the reason common factors were removed in \Cref{ABC notation} within the formulation of $\A$, $\B$, and $\C$.

\begin{prop}\label{Res tv and Implicit Equation}
With the polynomials $F_0,F_1 \in \K[t,v,w,x,y,z]$ in \Cref{Res su factors}, if $F_0\nsim F_1$ then their resultant $\Res_{tv} (F_0, F_1)$ is a nonzero polynomial in $\K[w,x,y,z]$. Moreover, this resultant factors as $\Res_{tv} (F_0, F_1) = EF$, for polynomials $E,F \in \K[w,x,y,z]$ where $F$ is the implicit equation of $X_\h$.
\end{prop}

\begin{proof}
Since $F_0$ and $F_1$ are irreducible, we note that $\gcd(F_0,F_1)=1$ if and only if $F_0\nsim F_1$, i.e. they do not differ by a unit multiple. In this case we have that $\Res_{tv} (F_0,F_1) \neq 0$ and, since both $F_0$ and $F_1$ follow $\h$, their resultant $\Res_{tv} (F_0,F_1)$ follows $\h$ as well. As such, there is some irreducible factor $F$ of $\Res_{tv} (F_0,F_1)$ following the parameterization $\h$. However, $F$ divides $\Res_{tv} (F_0,F_1)$ and is hence a polynomial in $\K[w,x,y,z]$, the coordinate ring of $\P^3$. The claim then follows as the only irreducible polynomial in $\K[w,x,y,z]$ and moving surface that follows $\h$ is the implicit equation of $X_\h$ itself.
\end{proof}

As noted, upon eliminating the variables $s,u,t,v$, any irreducible polynomial in $\K[w,x,y,z]$ that vanishes at $\h$ must be the implicit equation, since $X_\h$ is an irreducible surface in $\P^3$. In other words, the implicit equation is the moving surface following $\h$ that does not move! 

A potential obstruction to \Cref{Res tv and Implicit Equation} is that it is not clear when $F_0$ and $F_1$ are coprime, until computing them. The present author has been unable to produce any example of a translational surface $\h$ where these irreducible polynomials are associates of each other, and hence offers the following question as a potential inception for future research.

\begin{quest}\label{question F0 F1 coprime}
What conditions guarantee that $F_0\nsim F_1$ in \Cref{Res tv and Implicit Equation}? In other words, what condition ensures that $\Res_{su}(\x\A,\x\C)$ and $\Res_{su}(\x\B,\x\C)$ have distinct irreducible factors following $\h$?
\end{quest}

As noted in the previous section, there is a close connection to the surface $\h$ and its generating curves $\f$ and $\g$ in terms of syzygies and the resulting moving planes following $\h$. If $\Res_{tv}(F_0,F_1) =0$, there exists a point $(t_0,v_0) \in \P^1$ at which $F_0$ and $F_1$ vanish, after passing to the quotient field of $\K[w,x,y,z]$. It is not difficult to reduce to the case when $g_0(t_0,v_0)\neq 0$ (see e.g. the proof of \Cref{planar curve imp eqn theorem}), in which case $N_\g(t_0,v_0)$ is invertible. Noting that the columns of $N_\g \varphi_\f$ are multiples of the columns of $ [\A\,\B\,\C]$, one may reduce to the case of moving planes following the curve $\f$. As such, we ask a second question, an answer for which might lead to an answer to \Cref{question F0 F1 coprime}.

\begin{quest}\label{Curves res coprime}
Writing $\a$, $\b$, and $\c$ to denote the columns of $\varphi_\f$ as in (\ref{res of f}), is it possible for $\Res_{su}(\x\a,\x\c)$ and $\Res_{su}(\x\b,\x\c)$ to share a common factor? If so, what conditions on the curve $\f$ are sufficient to ensure that $\gcd(\Res_{su}(\x\a,\x\c), \Res_{su}(\x\b,\x\c)) =1$?
\end{quest}

We note that \Cref{Curves res coprime} has been addressed in \cite{HWJ10} for curves parameterized in low degree, however the author is unaware of a more general answer. 

We end this section with an example showing how \Cref{Res tv and Implicit Equation} may be employed to produce the implicit equation of $X_\h$. We borrow an example from \cite{GW18syz} and compare the effectiveness of our method with the algorithm developed there.

\begin{ex}[{\cite[3.6]{GW18syz}}]\label{Example used to compare 2 methods}
Consider the translational surface generated by the curves $\f:\, \P^1\dashrightarrow \P^3$ and $\g:\,\P^1\dashrightarrow \P^3$ given by
$$\f = (s^3,s^2u,su^2,u^3), \quad\quad\quad \g = (t^2,v^2,tv,tv).$$
The corresponding translational surface $\h:\,\P^1\times\P^1\dashrightarrow \P^3$ is given by 
$$\h=(2s^3t^2,\, s^3v+s^2ut^2,\,s^3tv+su^2t^2,\,s^3tv+u^3t^2).$$ 
With this, we compute
\[N_\g \varphi_\f = 
\left[
\begin{array}{cccc}
\tfrac{t^2}{2} & -\tfrac{v^2}{2} & -\tfrac{tv}{2} &-\tfrac{tv}{2}\\[0.5ex]
0&t^2&0&0\\
0&0&t^2&0\\
0&0&0&t^2
\end{array}
\right]
\left[
\begin{array}{ccc}
   u&0&0\\
   -s&u&0\\
   0&-s&u\\
   0&0&-s
\end{array}
\right]=\frac{1}{2}\left[
\begin{array}{ccc}
ut^2+sv^2 &stv-uv^2 &stv-utv\\
-2st^2&2ut^2&0\\
0&-2st^2&2ut^2\\
0&0&-2st^2
\end{array}
\right]_.
\]
Writing $\a$, $\b$, and $\c$ to denote the columns of $\varphi_\f$, we see that $\gcd(N_\g\a) = \gcd(N_\g\b) = 1$ and $\gcd(N_\g\c) = t$. Hence we may form the moving planes
\begin{align*}
\x\A &=ut^2w + sv^2w - 2st^2 x\\
\x\B &= stvw-uv^2w +2ut^2x-2st^2y \\
\x\C&= svw-uvw+2uty -2stz.
\end{align*}

Notice that, with respect to $s,u$, the polynomials above are linear, which also follows from \Cref{degrees ABC remark}. With this and \Cref{Res su nonzero}, we may formulate the resultants of these moving planes as the determinants of the $2\times 2$ Sylvester matrices
\begin{align*}
\Syl_{su}(\x\A,\x\C) &= \left[\begin{array}{cc}
v^2w-2t^2x &t^2w\\
vw-2tz &2ty-vw
\end{array}\right]\\[2ex]
\Syl_{su}(\x\B,\x\C) &= \left[\begin{array}{cc}
tvw -2t^2y& 2t^2x - v^2w  \\
vw-2tz &2ty-vw
\end{array}\right]_.
\end{align*}
Hence we have
\begin{align*}
\Res_{su}(\x\A,\x\C)&= -4t^3xy +2t^3wz - t^2vw^2  + 2t^2vwx+2tv^2wy - v^3w\\[1ex]
\Res_{su}(\x\B,\x\C)&= -4t^3y^2 +4t^3xz - 2t^2vwx+4t^2vwy-tv^2w^2-2tv^2wz+v^3w
\end{align*}
which are both irreducible polynomials. Thus $F_0=\Res_{su}(\x\A,\x\C)$ and $F_1=\Res_{su}(\x\B,\x\C)$ are the moving surfaces following $\h$ as in \Cref{Res su factors}, and are clearly not unit multiples of each other. Hence by \Cref{Res tv and Implicit Equation}, we have
\[
\Syl_{tv}(F_0,F_1) = \left[
\begin{array}{cccccc}
  2wz-4xy   & 2wx-w^2 &2wy &-w^2 &0&0 \\
  0& 2wz-4xy   & 2wx-w^2 &2wy &-w^2&0\\
  0&0&2wz-4xy   & 2wx-w^2 &2wy &-w^2\\
  4xz-4y^2 & 4wy-2wx &-w^2-2wz &w^2&0&0\\
  0& 4xz-4y^2 & 4wy-2wx &-w^2-2wz &w^2&0\\
  0&0& 4xz-4y^2 & 4wy-2wx &-w^2-2wz &w^2
\end{array}
\right]
\]
and its determinant is 
$$\Res_{tv}(F_0,F_1) = 8(y-z)w^5F$$
where 
\begin{align*}F =\,\,& 2w^4x^2-2w^3x^3-w^5y-w^4xy-12w^3x^2y+5w^4y^2+8w^3xy^2\\
&+12w^2x^2y^2-8w^2xy^3-20w^2y^4-24wxy^4+16y^6+w^5z+w^4xz\\
&+8w^3x^2z-8w^4yz-12w^3xyz+8w^3y^2z+48w^2xy^2z+24w^2y^3z\\
&+3w^4z^2-4w^3xz^2-14w^3yz^2-24w^2xyz^2-16wy^3z^2+6w^3z^3+4w^2z^4
\end{align*}
is the implicit equation of $X_\h$.
\end{ex}

Although $\Res_{su}(\x\A,\x\C)$ and $\Res_{su}(\x\B,\x\C)$ are irreducible in \Cref{Example used to compare 2 methods}, we note that it is common for the resultants involved to have extraneous factors. For instance, in the example above, $\Res_{tv}(F_0,F_1)$ has a factor of $w^5$ arising from basepoints (see discussion in \cite[3.4]{GW18syz} and its proof) and also the factor $y-z$ which appears precisely since the curve $\g$ lies on the plane $V(y-z)\subseteq\P^3$, which can be seen from its parameterization.

\begin{obs}\label{Comparison observation}
As noted, \Cref{Example used to compare 2 methods} coincides with the example given in \cite[3.6]{GW18syz}, and yields the implicit equation of $X_\h$ as well. However, we note that for the translational surface parameterized in bidegree $(3,2)$ above, the method of \cite{GW18syz} requires the determinant of a $12\times 12$ inhomogeneous Sylvester matrix. By comparison, the method presented here concerns the determinants of two $2\times 2$ matrices and a $6\times 6$ matrix, requiring less intermediate computations overall. Whereas the larger matrix of \cite[3.6]{GW18syz} is quite sparse, and so the difference in computation time is negligible, the disparity in size of the matrices involved between the two methods grows considerably for larger $m$ and $n$.
\end{obs}

As before, \Cref{Res tv and Implicit Equation} yields the implicit equation only if $\Res_{su}(\x\A,\x\C)$ and $\Res_{su}(\x\B,\x\C)$ have distinct factors following $\h$. In the proceeding section we show that, depending on the geometry of either curve, this criterion may be automatically satisfied, and the method developed here may be further simplified. In particular, for the surface $X_\h$ in \Cref{Example used to compare 2 methods}, the Sylvester matrices involved can be made to be considerably smaller, noting that the curve $\g$ is \textit{planar}, as it lies on the plane $V(y-z)\subseteq \P^3$.


\section{Planar curves and ruled surfaces}\label{Plane curves and ruled surfaces section}

As noted, the algorithm of the previous section produces the implicit equation of $X_\h$, provided that $\Res_{su}(\x\A,\x\C)$ and $\Res_{su}(\x\B,\x\C)$ have distinct factors following $\h$. In the current section, we show that this condition may be satisfied automatically in some cases, depending on how either of the generating curves embeds into $\P^3$. We consider the $\K$-vector space spanned by the homogeneous polynomials parameterizing one of the curves, say $\langle f_0,f_1,f_2,f_3\rangle$. In particular, the dimension of such a vector space will be a suitable invariant to distinguish between the cases.

\begin{rem}\label{dim remark}
With the assumptions of \Cref{main setting}, we have that $2\leq \dim_\K \f \leq 4$. Indeed, it is clear that $\dim_\K \f \leq 4$, and we note that $\dim_\K \f \neq 0,1$ as $f_0,f_1,f_2,f_3$ are not all identically zero and $\gcd \f =1$.
\end{rem}

Certainly one can repeat these observations for the vector space spanned by the polynomials of the curve $\g$. In particular, by symmetry, we may assume that $\dim_\K \f \leq \dim_\K \g$ throughout. Hence by \Cref{dim remark}, there is a natural way to divide \Cref{main setting} into cases, based on the dimension of the vector space spanned by $\f$. As such, we devote the rest of this section to the cases when $\dim_\K \f =2$ and $\dim_\K \f =3$, utilizing the geometry of the curves in this setting to show that the conditions of \Cref{Res tv and Implicit Equation} are automatically satisfied.

Notice that, since the polynomials of $\f$ are homogeneous, $\dim_\K \f$ coincides with the minimal number of generators of the ideal $(f_0,f_1,f_2,f_3)\subseteq \K[s,u]$. With this, $\dim_\K \f$ is the first Betti number appearing in a \textit{minimal} resolution of this ideal. Hence if $\dim_\K \f <4$, the resolution of $(\f)$ in \Cref{res f and g remark} must contain free summands, i.e. $\varphi_\f$ must contain unit entries. Following \Cref{syz lemma}, it will become apparent that $\h$ then has \textit{singly graded} syzygies, a phenomenon particularly useful in the implicitization of more general tensor product surfaces, as noted in \cite{DS16,SSV14,Weaver25}.

\subsection{Ruled surfaces} With the assumptions of \Cref{main setting}, we further assume that $\dim_\K \f =2$ throughout. In this case we show that $\f$ is a line in $\P^3$, hence $\h$ is a ruled translational surface, similar to the setting of \cite{GW18ruled}. We verify that the conditions of \Cref{Res tv and Implicit Equation} are satisfied, and that the method of \Cref{Iterated Resultant Section} can be simplified in this setting, and coincides with the implicitization method of \cite{GW18ruled} for this class of surfaces.

\begin{prop}\label{dim 2 line in P3}
With the assumptions of \Cref{main setting}, $\dim_\K \f=2$ if and only if $\f$ is a line in $\P^3$.  
\end{prop}

\begin{proof}
From the prior discussion, $\dim_\K \f=2$ if and only if the ideal of $\f$ has two minimal generators. Namely, the syzygy matrix $\varphi_\f$ in the resolution (\ref{res of f}) has a submatrix of rank two consisting of unit entries. As the degrees of the entries of $\varphi_\f$ are constant within each column, this resolution of $\f$ is thus
\begin{equation}\label{line in P3 res}
    0\rightarrow \begin{array}{c}
 A(-m)^2\\
 \oplus \\
A(-2m)\\
\end{array}\overset{\varphi_\f}{\longrightarrow}A(-m)^4 \rightarrow (\f)\rightarrow 0,
\end{equation}
where $A=\K[s,u]$, by \Cref{res f and g remark}. With $\a$, $\b$, and $\c$ the ordered columns of $\varphi_\f$, notice that the entries of $\a$ and $\b$ are units in $\K$ and the entries of $\c$ have degree $m$.

Recall that $\syz(\f)$ is free, hence the syzygies of $\f$ are linearly independent. Since the syzygies $\a$ and $\b$ in (\ref{line in P3 res}) have entries in $\K$ and $\f\a = \f\b = 0$, the curve $\f$ lies on the intersection of the planes $V(\x\a)$ and $V(\x\b)$. As these are distinct planes, we have that $\f$ is a line in $\P^3$. Conversely, if $\f$ is a line in $\P^3$, then it must lie on the intersection of two distinct planes. Thus there exist two linearly independent syzygies of $\f$ with unit entries, hence one has the non-minimal resolution above, following \Cref{res f and g remark}.
\end{proof}

As $\f$ is a line in $\P^3$, the surface $X_\h$ obtained by translating $\f$ along $\g$ is called a \textit{ruled} translational surface. These surfaces were studied extensively in \cite{GW18ruled} and a method to compute the implicit equation of $X_\h$ from a parameterization was given when $m=1$. Notice that if $m=1$ in \Cref{main setting}, then $\dim_\K \f =2$ automatically, following \Cref{dim remark} and noting that $\f$ is a $\K$-subspace of $\big[ \K[s,u]\big]_1 = \K s\oplus \K u$.

\begin{cor}\label{AB tv entries}
The syzygies $\A$ and $\B$ of $\h$ are linearly independent with entries in $\K[t,v]$.   
\end{cor}

\begin{proof}
This follows easily since $\a$ and $\b$ have of entries in $\K$ and the entries of $N_\g$ belong to $\K[t,v]$. Moreover, recall that $\gcd(N_\g \a)$ and $\gcd(N_\g \b)$ are polynomials in $\K[t,v]$ by \Cref{gcds divide g0}. Lastly, the independence follows from \Cref{A B C linearly independent}. 
\end{proof}

We now show that the conditions of \Cref{Res tv and Implicit Equation} are satisfied when $\dim_\K \f =2$, and the method of \Cref{Iterated Resultant Section} coincides with the main result of \cite[3.1]{GW18ruled} in this case.

\begin{prop}\label{ruled surface prop}
   With the assumptions of \Cref{main setting}, further assume that $\dim_\K \f =2$. The resultant $\Res_{tv}(\x\A,\x\B)$ is a nonzero polynomial in $\K[w,x,y,z]$ and factors as $\Res_{tv}(\x\A,\x\B)=EF$, where $F$ is the implicit equation of $X_\h$.
\end{prop}

\begin{proof}
Recall from \Cref{AB tv entries} that $\A$ and $\B$ are syzygies of $\h$ with entries in $\K[t,v]$, hence the moving planes $\x\A$ and $\x\B$ are polynomials in $\K[t,v,w,x,y,z]$. From \Cref{degrees ABC remark} and (\ref{line in P3 res}), we see that $\x\C$ has degree $m$ with respect to $s,u$, and both $\x\A$ and $\x\B$ have degree $0$ with respect to $s,u$. Thus from the construction of the Sylvester matrix (\ref{Sylvester matrix definition}), it follows that $\Res_{su}(\x\A,\x\C) = (\x\A)^m$ and $\Res_{su}(\x\B,\x\C) = (\x\B)^m$. Hence the irreducible factors of these resultants following $\h$ in \Cref{Res su factors} are precisely $F_0 = \x\A$ and $F_1=\x\B$. The result now follows from \Cref{Res tv and Implicit Equation}, noting that $\x\A$ and $\x\B$ are coprime by \Cref{xA xB xC irreducible and coprime}.
\end{proof}

In particular, \Cref{ruled surface prop} coincides with the method of \cite[3.1]{GW18ruled} to obtain the implicit equation of a ruled translational surface $X_\h$, again noting that $\dim_\K \f = 2$ automatically if $m=1$.

\subsection{Planar curves}
With the case of ruled translational surfaces when $\dim_\K \f =2$ settled, we now consider the case when $\dim_\K \f =3$. As such, $\f$ is a \textit{planar} curve in $\P^3$, following a similar argument to the proof of \Cref{dim 2 line in P3}. This setting is of particular interest for applications in geometric modeling as one commonly considers translational surfaces generated by curves $\f$ and $\g$ lying in orthogonal planes \cite{XXR12}.

\begin{prop}\label{special syzs planar curve}
With the assumptions of \Cref{main setting}, if $\dim_\K \f=3$ then the curve $\f$ lies on a plane in $\P^3$. 
\end{prop}

\begin{proof}
We proceed as in the proof \Cref{dim 2 line in P3}. As $\dim_\K \f = 3$, the ideal of $\f$ is minimally generated by $3$ elements. Hence the syzygy matrix in the free resolution (\ref{res of f}), with respect to the generating set $f_0,f_1,f_2,f_3$, must have a single column with unit entries. Hence $\mu_1=0$ and this non-minimal resolution of $\f$ is
\begin{equation}\label{planar curve res}
0\rightarrow \begin{array}{c}
 A(-m)\\
 \oplus \\
A(-m-\mu_2)\\
 \oplus \\
A(-m-\mu_3)\\
\end{array}\overset{\varphi_\f}{\longrightarrow}A(-m)^4 \rightarrow (\f)\rightarrow 0,
\end{equation}
where $A=\K[s,u]$ and $\mu_2+\mu_3 =m$. With $\a$, $\b$, and $\c$ the ordered columns of $\varphi_\f$, the entries of $\a$ are units in $\K$ and the entries of $\b$ and $\c$ have degree $\mu_2$ and $\mu_3$ respectively. Since $\a$ has entries in $\K$ and $\f\a = 0$, it follows that the curve $\f$ lies on the plane $V(\x\a)\subseteq\P^3$.
\end{proof}

As noted in \Cref{AB tv entries}, since the syzygy $\a$ of $\f$ has entries in $\K$, the syzygy $\A$ of $\h$ has entries in $\K[t,v]$. Hence we may proceed as in the proof of \Cref{ruled surface prop} for the resultants involving the moving plane $\x\A$.

\begin{thm}\label{planar curve imp eqn theorem}
With the assumptions of \Cref{main setting}, assume that $\dim_\K \f =3$. Writing $F_1$ to denote the factor of $\Res_{su}(\x\B,\x\C)$ following $\h$ in \Cref{Res su factors}, the resultant $\Res_{tv}(\x\A, F_1)$ is a nonzero polynomial in $\K[w,x,y,z]$ and factors as $\Res_{tv}(\x\A, F_1)=EF$, where $F$ is the implicit equation of $X_\h$.
\end{thm}

\begin{proof}

As noted in the proof of \Cref{special syzs planar curve}, the syzygy $\a$ in (\ref{planar curve res}) consists of entries in $\K$. As such, the syzygy $\A$ of $\h$ has entries in $\K[t,v]$. Following \Cref{degrees ABC remark}, we thus have that $\Res_{su}(\x\A,\x\C) = (\x\A)^{\mu_3}$ where $\mu_3\geq 1$ is as in (\ref{planar curve res}), and so $F_0 = \x\A$ is the only irreducible factor of this resultant following $\h$. Thus the claim will follow from \Cref{Res tv and Implicit Equation} once it has been shown that $\gcd (\x\A, F_1) =1$ for $F_1$ in \Cref{Res su factors}. Suppose otherwise, and assume that $\x\A$ divides $\Res_{su}(\x\B,\x\C)$.

Notice that for any point $(t_0,v_0,w_0,x_0,y_0,z_0)$ on $V(\x\A)\subseteq \P^1\times \P^3$, since $\x\A$ divides $\Res_{su}(\x\B,\x\C)$, this resultant vanishes at this point as well. Hence there exists a point $(s_0,u_0)\in \P^1$ such that both $\x\B$ and $\x\C$ vanish at $(s_0,u_0,t_0,v_0,w_0,x_0,y_0,z_0)$. Since $\x\A$ is a polynomial in $\K[t,v,w,x,y,z]$, we note that it also vanishes at this point in $\P^1\times\P^1\times\P^3$.

We claim that we may assume that $g_0(t_0,v_0) \neq 0$ for such a point above. As we need only show that such a point on $V(\x\A)$ with these coordinates and this condition exists, suppose to the contrary that $g_0(t_0,v_0) =0$ for all points $(t_0,v_0,w_0,x_0,y_0,z_0)\in V(\x\A)$. As $N_\g\a$, $N_\g\b$, and $N_\g\c$ are multiples of $\A$, $\B$, and $\C$, and the moving planes $\x\A$, $\x\B$, and $\x\C$ vanish at the point above, following (\ref{Nga description}) we have the following system
\begin{equation}\label{system of eqns}
\arraycolsep=1.4pt
\left\{
\begin{array}{lcc}
(g_0a_0 -g_1a_1-g_2a_2-g_3a_3)w_0&=&0  \\[1ex]
(g_0b_0 -g_1b_1-g_2b_2-g_3b_3)w_0&=&0  \\[1ex]
(g_0c_0 -g_1c_1-g_2c_2-g_3c_3)w_0&=&0 
\end{array}
\right.
\end{equation}
when evaluated at $(s_0,u_0,t_0,v_0)$, where the coefficients are the entries of $\a$, $\b$, and $\c$. 

With (\ref{system of eqns}) above, notice that if $w_0 =0$, then $(w_0,x_0,y_0,z_0)$ is a point at infinity in $\P^3$. If $w_0\neq 0$ then, since $g_0(t_0,v_0)=0$ by assumption, we have the following matrix equation
\begin{equation}\label{matrix eqn}
\left[\begin{array}{ccc}
     a_1&a_2&a_3  \\[1ex]
     b_1&b_2&b_3  \\[1ex]
     c_1&c_2&c_3  
\end{array}\right]_{(s_0,u_0)}\left[\begin{array}{c}
     g_1  \\[1ex]
     g_2 \\[1ex]
     g_3 
\end{array}\right]_{(t_0,v_0)} =0.
\end{equation}
Since $g_0(t_0,v_0) =0$ and $\g$ is a curve, and hence has no basepoints as $\gcd \g =1$, it follows that at least one of $g_1$, $g_2$, $g_3$ is non-vanishing at $(t_0,v_0)$. As such, it follows that the determinant of the $3\times 3$ matrix in (\ref{matrix eqn}) must vanish. However, by \Cref{res f and g remark}, this determinant is precisely $f_0(s_0,u_0)$. Since both $f_0(s_0,u_0) =0$ and $g_0(t_0,v_0) =0$, by \Cref{basepoint lemma} we have that $(s_0,u_0,t_0,v_0)$ is a basepoint of $\h$.

From the discussion above, it follows that if $g_0(t_0,v_0)=0$ for every point $(t_0,v_0,w_0,x_0,y_0,z_0)\in V(\x\A)$, every such point on $V(\x\A)$ corresponds either to a basepoint of $\h$ or a point at infinity in $\P^3$. However, this is impossible as $\x\A$ divides $\Res_{su}(\x\B,\x\C)$ and follows $\h$, and so it cannot be an extraneous factor of this resultant arising from basepoints or anomalies at infinity \cite{SG17}. As such, we may assume that $g_0(t_0,v_0) \neq 0$ for the point $(t_0,v_0,w_0,x_0,y_0,z_0)\in V(\x\A)$ above, as such a point must exist.

With this, fix $(t_0,v_0)$ and consider the plane $V\big(\x\A(t_0,v_0)\big)\subseteq \P^3$, recalling that $\A$ has entries in $\K[t,v]$. Since $\x\A$ divides $\Res_{su}(\x\B,\x\C)$, repeating the previous argument shows that for any point $(w_0,x_0,y_0,z_0)\in V\big(\x\A(t_0,v_0)\big)$, there exists a point $(s_0,u_0)$ such that $\x\A$, $\x\B$, and $\x\C$ simultaneously vanish at the point $(s_0,u_0,t_0,v_0,w_0,x_0,y_0,z_0)$ in $\P^1\times\P^1\times \P^3$. Since $g_0(t_0,v_0) \neq 0$, by \Cref{Intersection xA and xB and xC} it follows that the point $(w_0,x_0,y_0,z_0)$ lies on the surface $X_\h$. In particular, it follows that \textit{every} point on the plane $V\big(\x\A(t_0,v_0)\big)\subseteq \P^3$ must lie on the surface $X_\h$. However, this is only possible if $V\big(\x\A(t_0,v_0)\big)$ is the surface $X_\h$ itself, and so $\x\A(t_0,v_0)$ is the implicit equation. However, $\h$ parameterizes a translational surface generated by curves $\f$ and $\g$, and $X_\h=V\big(\x\A(t_0,v_0)\big)$ is a plane in $\P^3$. This is only possible if both $\f$ and $\g$ are lines in $\P^3$, which is a contradiction by \Cref{dim 2 line in P3}, as $\dim_\K\f =3$. 
\end{proof}

We note that one could have assumed instead that $\dim_\K \f \leq 3$ in an effort to combine \Cref{ruled surface prop} and \Cref{planar curve imp eqn theorem}. In this case, the proof above proceeds similarly, and one has that both $\f$ and $\g$ must be lines in $\P^3$, hence $\dim_\K \f =2$ by \Cref{dim 2 line in P3}. However, the contradiction then follows from the proof of \Cref{ruled surface prop} as $\Res_{su}(\x\B,\x\C) = (\x\B)^m$, and so it is impossible for $\x\A$ to divide $\Res_{su}(\x\B,\x\C)$ following \Cref{xA xB xC irreducible and coprime}.

\begin{rem}
Notice that the proof of \Cref{planar curve imp eqn theorem} shows that $\x\A$ does not divide $\Res_{su}(\x\B,\x\C)$. Since $\x\A$ is irreducible by \Cref{xA xB xC irreducible and coprime}, one could alternatively compute $\Res_{tv}(\x\A,\Res_{su}(\x\B,\x\C))$ as this is hence nonzero, and so the implicit equation $F$ must be among its factors. However, it is simpler to remove unnecessary factors of $\Res_{su}(\x\B,\x\C)$ first, following the resultant property in (\ref{resultant product property}). 
\end{rem}

\begin{rem}
As noted prior to \Cref{AB tv entries}, we have that $\dim_\K \f =2$ automatically if $m=1$. Similarly, it follows that $\dim_\K \f \leq 3$ if $m =2$, as then $\f$ is a $\K$-subspace of $\big[ \K[s,u]\big]_2 = \K s^2\oplus \K su \oplus \K u^2$. In particular, \Cref{ruled surface prop} and \Cref{planar curve imp eqn theorem} may be employed to compute the implicit equation for any translational surface $X_\h$ parameterized in bidegree $(m,n)$ when either $m\leq 2$ or $n\leq 2$, by symmetry.
\end{rem}

To demonstrate the effectiveness of \Cref{planar curve imp eqn theorem}, we consider a translational surface $X_\h$ with parameterization $\h:\,\P^1\times\P^1\dashrightarrow\P^3$ that has a basepoint as in \Cref{bad basepoint}. As such, the method of \cite[3.4]{GW18syz} cannot be applied. However, we show that \Cref{planar curve imp eqn theorem} succeeds in producing the implicit equation, requiring relatively small matrices in its computations.

\begin{ex}[{\cite[3.7]{GW18syz}}]\label{Planar curve example}
Consider the translational surface generated by the curves $\f:\, \P^1\dashrightarrow \P^3$ and $\g:\,\P^1\dashrightarrow \P^3$ given by
$$\f = (s^3,0,s^2u,u^3), \quad\quad\quad \g = (t^2,tv,0,-v^2).$$
The corresponding translational surface $\h:\,\P^1\times\P^1\dashrightarrow \P^3$ is given by 
$$\h= (2s^3t^2, s^3tv, s^2ut^2, -s^3v^2+u^3t^2).$$
As noted in \cite[3.7]{GW18syz}, the parameterization $\h$ has only one basepoint $p=(1,0,1,0) \in \P^1\times \P^1$ with multiplicity $9$. Additionally, since $\f(1,0) = (1,0,0,0) =-\g(1,0)$, we see that $p$ is a basepoint as in \Cref{bad basepoint}, and so the implicitization method of \cite[3.4]{GW18syz} fails. However, we show that \Cref{planar curve imp eqn theorem} does succeed in producing the implicit equation of $\h$.

We note that $\f$ is a planar curve as it lies on the plane $V(x)\subseteq \P^3$, which can be seen from its parameterization. Computing $N_\g\varphi_\f$, we have
\[
N_\g \varphi_\f = 
\left[
\begin{array}{cccc}
   \frac{t^2}{2} &-\frac{tv}{2}&0&\frac{v^2}{2}\\[0.5ex]
   0&t^2&0&0\\
   0&0&t^2&0\\
   0&0&0&t^2
\end{array}
\right]
\left[
\begin{array}{ccc}
   0&-u&0\\
   1&0&0\\
   0&s&-u^2\\
   0&0&s^2
\end{array}
\right]=
\frac{1}{2}
\left[
\begin{array}{ccc}
  -tv&-ut^2&s^2v^2\\
  2t^2&0&0\\
  0&2st^2&-2u^2t^2\\
  0&0&2s^2t^2
\end{array}
\right]_.
\]
Writing $\a$, $\b$, and $\c$ to denote the columns of $\varphi_\f$ above, we see that $\gcd(N_\g\a) = t$, $\gcd(N_\g\b) = t^2$, and $\gcd(N_\g\c) = 1$. Hence after dividing each column by the greatest common divisor of its entries, we may form the moving planes
\begin{align*}
\x\A &=-vw+2tx\\
\x\B &= -uw+2sy \\
\x\C&= s^2v^2w-2u^2t^2y + 2s^2t^2z.
\end{align*}

From the equations above, and also following \Cref{degrees ABC remark}, we note that $\x\B$ is a polynomial of degree $1$, with respect to $s,u$. Similarly, $\x\C$ is a polynomial of degree $2$, with respect to $s,u$. Hence their Sylvester matrix is the $3\times 3$ matrix 
\[
\Syl_{su}(\x\B,\x\C) = \left[
\begin{array}{ccc}
   2y  &-w&0  \\
    0 & 2y&-w\\
    v^2w +2t^2z &0&-2t^2y
\end{array}
\right]_.
\]
Thus we have 
$$\Res_{su}(\x\B,\x\C) = \det \Syl_{su}(\x\B,\x\C)= 2t^2w^2z-8t^2y^3 +v^2w^3$$
which is irreducible. With this, we may compute the second Sylvester matrix, with respect to $t,v$, noting that $\x\A$ is linear in $t,v$ and the polynomial above has degree $2$ with respect to $t,v$. Hence we obtain the $3\times 3$ Sylvester matrix
\[
\Syl_{tv}(\x\A,\Res_{su}(\x\B,\x\C)) = \left[
\begin{array}{ccc}
2x&-w&0\\
0&2x&-w\\
2w^2z-8y^3&0&w^3
\end{array}
\right]
\]
and so
$$\Res_{tv}(\x\A,\Res_{su}(\x\B,\x\C))= \det \Syl_{tv}(\x\A,\Res_{su}(\x\B,\x\C)) =2w^2(2wx^2-4y^3+w^2z).$$
Analyzing the factors of this resultant, it follows that the implicit equation of the translational surface $X_\h$ is $F= 2wx^2-4y^3+w^2z$.
\end{ex}

We note that the curve $\g$ in \Cref{Planar curve example} is planar as well, as it lies in the plane $V(y) \subseteq \P^3$. Hence one could alternatively compute $N_\f \varphi_\g$ and proceed as above. In this case, one encounters a $2\times 2$ Sylvester matrix with respect to $t,v$, and then a $4\times 4$ Sylvester matrix with respect to $s,u$, to obtain the implicit equation. As mentioned, this setting of translational surfaces constructed from curves lying in orthogonal planes is common within geometric modeling \cite{XXR12}, hence \Cref{planar curve imp eqn theorem} provides a simple solution to the implicitization problem in a common setting for translational surfaces.

\section*{Acknowledgements}

The use of \texttt{Macaulay2} \cite{Macaulay2} was helpful in the preparation of this article, offering numerous examples to support the results presented here.



\begin{thebibliography}{99}


\bibitem{AHW05} W.~A. Adkins, J.~W. Hoffman, and H. Wang \textit{Equations of parametric surfaces with base points via syzygies}, J. Symbolic Comput. \textbf{39} (2005),  73--101. 


\bibitem{BPF26}  R. M. Benedicto, S. P\'erez-D\'iaz, and M. Fern\'andez de Sevilla, \textit{Implicitization of 3{D} rational surfaces using iterated univariate resultants and avoiding extraneous factors}, Rev. R. Acad. Cienc. Exactas F\'is. Nat. Ser. A Mat. RACSAM \textbf{120} (2026), Paper No. 22, 20.


\bibitem{Botbol11}  N. Botbol, \textit{The implicit equation of a multigraded hypersurface}, J. Algebra \textbf{348} (2011), 381--401.











\bibitem{BM09} L. Bus\'e and B. Mourrain, \textit{Explicit factors of some iterated resultants and discriminants}, Math. Comp. \textbf{78} (2009),  345--386.







 






\bibitem{Cox08} D. Cox, \textit{The moving curve ideal and the Rees algebra},  Theoret. Comput. Sci. \textbf{392} (2008), 23--36. 


\bibitem{CGZ00} D. Cox, R. Goldman, M. Zhang, \textit{On the
validity of implicitization by moving quadrics for rational surfaces
with no basepoints}, J. Symbolic Comput. \textbf{29} (2000), 419--440.


\bibitem{CLO} D. Cox, J. Little, and D. O'shea, \textit{Using algebraic geometry}, Graduate Texts in Mathematics \textbf{185}, Springer, New York, 2005. 


\bibitem{DS16} E. Duarte and H. Schenck, \textit{Tensor product surfaces and linear syzygies}, Proc. Amer. Math. Soc. \textbf{144} (2016), 65--72.



\bibitem{Eisenbud} D.~Eisenbud, \textit{Commutative algebra: with a view toward algebraic geometry}, Graduate Texts in Mathematics \textbf{150}, Springer-Verlag, New York, 1995. 

\bibitem{GKZ} I.M. Gelfand, M.M. Kapranov, and A.V. Zelevinsky, \textit{Discriminants, resultants, and multidimensional determinants}, Mathematics: Theory \& Applications, Birkh\"auser, Boston, 1994. 


\bibitem{GW18syz} R. Goldman and H. Wang, \textit{Syzygies for translational surfaces}, J. Symbolic Comput. \textbf{89} (2018), 73--93.


\bibitem{GW18ruled} R. Goldman and H. Wang, \textit{Implicitizing ruled translational surfaces}, Comput. Aided Geom. Des. \textbf{59} (2018), 98--106.

  

\bibitem{Macaulay2} D. R. Grayson and M. E. Stillman, Macaulay2, a software system for research in algebraic geometry. Available at http://www.math.uiuc.edu/Macaulay2/






\bibitem{HWJ10} J. Hoffman, H. Wang, and X. Jia, \textit{Minimal generators for the {R}ees algebra of rational space curves of type {$(1,1,d-2)$}}, Eur. J. Pure Appl. Math. \textbf{3} (2010), 602--632.



\bibitem{Liu99} H. Liu, \textit{Translation surfaces with constant mean curvature in {$3$}-dimensional spaces}, J. Geom. \textbf{64} (1999), 141--149.

\bibitem{LV17} M. L\'avi\v cka and J. Vr\v sek, \textit{Translation surfaces and isotropic transport nets on rational minimal surfaces}, Mathematical methods for curves and surfaces, Lecture Notes in Comput. Sci. \textbf{10521} (2017), 186--201.
  

\bibitem{LM09} D. Lazard and S. McCallum, \textit{Iterated discriminants}, J. Symbolic Comput. \textbf{44} (2009), 1176--1193.

\bibitem{McCallum99} S. McCallum, \textit{Factors of iterated resultants and discriminants}, J. Symbolic Comput. \textbf{27} (1999), 367--385.


\bibitem{PDS20} S. P\'erez-D\'iaz and L. Shen, \textit{Parameterization of rational translational surfaces}, Theoret. Comput. Sci. \textbf{835} (2020), 156--167.
  

\bibitem{SSV14} H. Schenck, A. Seceleanu, and J. Validashti, \textit{Syzygies and singularities of tensor product surfaces of bidegree $(2,1)$}, Math. Comp. \textbf{83} (2014), 1337--1372.


\bibitem{SC95} T. W. Sederberg and F. Chen, \textit{Implicitization
using moving curves and surfaces}, Proceedings of
SIGGRAPH (1995), 301--308.

\bibitem{SGD97} T. W. Sederberg, R. N. Goldman, and H. Du, \textit{Implicitizing rational curves by the method of moving algebraic
curves}, J. Symb. Comput. \textbf{23} (1997), 153--175.

\bibitem{SSQK94} T. W. Sederberg, T. Saito,
D. Qi, and K. S. Klimaszewksi, \textit{Curve implicitization using moving lines}, Comput. Aided Geom. Des. \textbf{11} (1994),
687--706.


\bibitem{SG17} L. Shen and R. Goldman, \textit{Strong {$\mu$}-bases for rational tensor product surfaces and extraneous factors associated to bad base points and anomalies at infinity}, SIAM J. Appl. Algebra Geom. \textbf{1} (2017), 328--351.


\bibitem{SG18} L. Shen and R. Goldman, \textit{Combining complementary methods for implicitizing rational tensor product surfaces}, Comput.-Aided Des. \textbf{104} (2018), 100--112.

\bibitem{XXR12} X. Shi, X. Wang, and R. Goldman, \textit{Using {$\mu$}-bases to implicitize rational surfaces with a pair of orthogonal directrices}, Comput. Aided Geom. Design \textbf{29} (2012), 541--554.

\bibitem{Weaver25} M. Weaver, \textit{Tensor product surfaces and quadratic syzygies}, Linear Algebra Appl. \textbf{720} (2025), 350--371.


 
\end{thebibliography}
\end{document}